\documentclass[a4paper,twoside,10pt,reqno]{article}
\usepackage[english]{babel}
\usepackage{graphicx}
\usepackage{epstopdf}
\usepackage[T1]{fontenc}
\usepackage[latin1]{inputenc}
\usepackage{amssymb,amsmath,amsthm,dsfont,mathrsfs}
\usepackage{amsfonts,euscript}
\usepackage{color}
\usepackage{authblk}

\usepackage{authblk}
\usepackage{fullpage}
\usepackage{setspace}
\newtheorem{teo}{Theorem}[section]
\newtheorem{pro}[teo]{Proposition}
\newtheorem{cor}[teo]{Corollary}

\newtheorem{ex}{Example}[section]

\newcommand{\ze}{\mathbb{Z}}
\newcommand{\er}{\mathbb{R}}

\newcommand{\dst}{\displaystyle}

\newcommand{\nn}{\nonumber}
\newcommand{\nid}{\noindent }
\newcommand{\mb}{\mathbf }
\newcommand{\ea}{\emph{et al.}}
\newcommand{\es}{\left}
\newcommand{\di}{\right}
\newcommand{\mei}{multivariate extremal index}

\begin{document}

\title{Dissecting the multivariate extremal index and tail dependence}
\author{Helena Ferreira}
\affil{Universidade da Beira Interior, Centro de Matem\'{a}tica e Aplica\c{c}\~oes (CMA-UBI), Avenida Marqu\^es d'Avila e Bolama, 6200-001 Covilh\~a, Portugal\\ \texttt{helena.ferreira@ubi.pt}}

\author{Marta Ferreira}
\affil{Center of Mathematics of Minho University\\ Center for Computational and Stochastic Mathematics of University of Lisbon \\
Center of Statistics and Applications of University of Lisbon, Portugal\\ \texttt{msferreira@math.uminho.pt} }

\date{}

\maketitle

\abstract{A central issue in the theory of extreme values focuses on suitable conditions such that the well-known results for the limiting distributions of the maximum of i.i.d. sequences can be applied to stationary ones.
In this context, the extremal index appears as a key parameter to capture the effect of temporal dependence on the limiting distribution of the maxima. The  multivariate extremal index corresponds to a generalization of this concept to a multivariate context and affects the tail dependence structure within the marginal sequences and between them. As it is a function, the inference becomes more difficult, and it is therefore important to obtain characterizations, namely bounds based on the marginal dependence that are easier to estimate. In this work we present two decompositions that emphasize different types of information contained in the multivariate extremal index, an upper limit better than those found in the literature and we analyze its role in dependence on the limiting model of the componentwise maxima of a stationary sequence. We will illustrate the results with examples of recognized interest in applications.}\\

\nid\textbf{keywords:} {multivariate extreme values, multivariate extremal index, tail dependence, extremal coefficients, madogram}\\

\nid\textbf{AMS 2000 Subject Classification}: 60G70\\

\section{Introduction}\label{sint}
Let $F$ be a multivariate distribution function (df), with continuous marginal dfs, in the domain of attraction of a multivariate extreme values (MEV) df $\widehat{H}$ having unit Fr\'{e}chet marginals. Thus
\begin{eqnarray}\label{domatrac1}
F^n(nx_1,\hdots,nx_d)\to \widehat{H}(x_1,\hdots,x_d)
\end{eqnarray}
and $\widehat{H}_j(x_j)=\exp(-x_j^{-1})$, $x_j>0$.

Consider $\{\mb{X}_n=(X_{n1},\hdots,X_{nd})\}$ a stationary sequence such that $F_{\mb{X}_n}=F$ and let $\{\mb{M}_n=(M_{n1},\hdots,M_{nd})\}$ be a componentwise maxima sequence generated from $\mb{X}_1,\hdots,\mb{X}_n$ and therefore $M_{nj}=\bigvee_{i=1}^nX_{ij}$, $j=1,\hdots,d$. If
\begin{eqnarray}\label{domatrac2}
\dst\lim_{n\to\infty}P(M_{n1}\leq nx_1,\hdots,M_{nd}\leq nx_d)={H}(x_1,\hdots,x_d),
\end{eqnarray}
for some MEV df $H$, we can relate ${H}(x_1,\hdots,x_d)$ and $\widehat{H}(x_1,\hdots,x_d)$ through the so called multivariate extremal index of $\{\mb{X}_n\}$. This is possible, even if the marginals $\widehat{H}_j$ are not unit Fr\'{e}chet distributed, as considered for simplicity and without loss of generality. Indeed, to have (\ref{domatrac1}) or {\it mutatis mutandis} (\ref{domatrac2}), it is sufficient that, as $n\to\infty$, the sequence of copulas $C_F^n$, with $C_F(u_1,\hdots,u_d)=F(F^{-1}_1(u_1),\hdots,F^{-1}_1(u_d))$, converges to $C_{\widehat{H}}$, as well as, $F_j^n(nx_j)\to \widehat{H}_j(x_j)$, $j=1,\hdots,d$, which can be reduced to the case of convergence to the Fr\'{e}chet without affecting the convergence of $C_F^n$.

We recall the definition of multivariate extremal index of $\{\mb{X}_n\}$ and its role in the relation between ${H}$ and $\widehat{H}$ (Nandagopalan \cite{nand94} 1994). The sequence $\{\mb{X}_n\}$ has multivariate extremal index $\theta(\pmb{\tau})\in(0,1]$, $\pmb{\tau}=(\tau_1,\hdots,\tau_d)\in\er^d_+$, when for each $\pmb{\tau}$ there is a sequence of real levels $\{\mb{u}_n^{(\pmb{\tau})}=(u_{n1}^{(\tau_1)},\hdots, u_{nd}^{(\tau_d)})\}$ satisfying
\begin{eqnarray}\nn
\begin{array}{c}
\dst nP(X_{1j}>u_{nj}^{(\tau_j)})\to \tau_j,\,j\in D=\{1,\hdots,d\},\\\\
\dst P(\mb{\widehat{M}}_n\leq \mb{u}_n^{(\pmb{\tau})})\to\widehat{\gamma}(\pmb{\tau}) \textrm{ and}\\\\
\dst P(\mb{{M}}_n\leq \mb{u}_n^{(\pmb{\tau})})\to{\gamma}(\pmb{\tau})=(\widehat{\gamma}(\pmb{\tau}))^{\theta(\pmb{\tau})},
\end{array}
\end{eqnarray}
where $\mb{\widehat{M}}_n=(\widehat{M}_{n1},\hdots,\widehat{M}_{nd})$, $\widehat{M}_{nj}=\bigvee_{i=1}^n\widehat{X}_{ij}$, $j=1,\hdots,d$, and $\{\mb{\widehat{X}}_n\}$ is a sequence of independent vectors such that $F_{\mb{\widehat{X}}_n}=F_{\mb{X}_n}$.

Observe that
\begin{eqnarray}\nn
\begin{array}{c}
\dst \widehat{\gamma}(\pmb{\tau})=\exp\es(-\lim_{n\to\infty} nP(\mb{X}_1\not\leq\mb{u}_n)\di)
=\exp\es(-\Gamma(\pmb{\tau})\di),
\end{array}
\end{eqnarray}
with
\begin{eqnarray}\nn
\begin{array}{rl}
\dst \Gamma(\pmb{\tau})=&\dst \lim_{n\to\infty} nP\es(\bigcup_{j=1}^d\{X_{ij}>u_{nj}^{(\tau_j)}\}\di)\\\\
=&\dst \sum_{\emptyset\not= J\subset D}(-1)^{|J|+1}\lim_{n\to\infty}nP\es(\bigcap_{j\in J}\{X_{ij}>u_{nj}^{(\tau_j)}\}\di)\\\\
=&\dst \sum_{\emptyset\not= J\subset D}(-1)^{|J|+1}\Gamma^*_J(\pmb{\tau}_J),
\end{array}
\end{eqnarray}
where
\begin{eqnarray}\nn
\begin{array}{c}
\dst \Gamma^*_J(\pmb{\tau}_J)\equiv \Gamma^*(\tau_j,\,j\in J)=\lim_{n\to\infty}nP\es(\bigcap_{j\in J}\{X_{ij}>u_{nj}^{(\tau_j)}\}\di)
\end{array}
\end{eqnarray}
and, in particular, $\Gamma^*_{\{j\}}(\tau_j)=\tau_j $, $j\in D$. So, to say that $\Gamma(\pmb{\tau})$ exists is equivalent to say that  $\widehat{\gamma}(\pmb{\tau})$ exists and we have
\begin{eqnarray}\nn
\begin{array}{c}
\dst {\gamma}(\pmb{\tau})=\exp\es(-\theta(\pmb{\tau})\Gamma(\pmb{\tau})\di)
=\exp\es(-\theta(\pmb{\tau})\sum_{\emptyset\not= J\subset D}(-1)^{|J|+1}\Gamma^*_J(\pmb{\tau}_J)\di),
\end{array}
\end{eqnarray}

If $\{\mb{X}_n\}$ has multivariate extremal index $\theta(\pmb{\tau})$ then any sequence of subvectors $\{(\mb{X}_n)_A\}$ with indexes in $A\subset\{1,\hdots,d\}$ has multivariate extremal index $\theta_A(\pmb{\tau}_A)$, with
\begin{eqnarray}\nn
\begin{array}{c}
\dst \theta_A(\pmb{\tau}_A)=\lim_{\substack{\tau_i\to 0^+\\ i\not\in A}}\theta(\tau_1,\hdots,\tau_d),\,\pmb{\tau}_A\in\er_+^{|A|}\,.
\end{array}
\end{eqnarray}
In particular, for each $j=1,\hdots,d$, $\{X_{nj}\}_{n\geq 1}$ has extremal index $\theta_j$.

If $\theta(\pmb{\tau})$, $\pmb{\tau}\in\er_+^{d}$, exists for $\{\mb{X}_n\}$ we have
\begin{eqnarray}\label{HeHchap}
\begin{array}{c}
\dst H(x_1,\hdots,x_d)=\widehat{H}(x_1,\hdots,x_d)^{\theta(-\log \widehat{H}_1(x_1),\hdots,-\log \widehat{H}_d(x_d))}
\end{array}
\end{eqnarray}
and $H_j(x_j)=\widehat{H}_j(x_j)^{\theta_j}$, $j\in D$.

From inequalities
\begin{eqnarray}\nn
\begin{array}{c}
\dst \prod_{j=1}^d\widehat{H}_j(x_j)^{\theta_j}\leq \widehat{H}(x_1,\hdots,x_d)^{\theta(\tau_1(x_1),\hdots,\tau_d(x_d))}\leq \min_{j=1,\hdots,d}\widehat{H}_j(x_j)^{\theta_j},
\end{array}
\end{eqnarray}
we obtain
\begin{eqnarray}\label{IEMlims}
\begin{array}{c}
\dst \frac{\bigvee_{j=1}^d\theta_j\tau_j}{\Gamma(\pmb{\tau})}\leq \theta(\pmb{\tau})\leq \frac{\sum_{j=1}^d\theta_j\tau_j}{\Gamma(\pmb{\tau})}\,.
\end{array}
\end{eqnarray}

Besides the relation between $H$ and $\widehat{H}$, $\theta(\pmb{\tau})$ also informs about the existence of clustering of events ``at least some exceedance of $u_{nj}^{(\tau_j)}$ by $X_{nj}$, for some $j$", since
\begin{eqnarray}\label{IEMclustmean}
\begin{array}{c}
\dst \frac{1}{\theta(\pmb{\tau})}=\dst\lim_{n\to\infty}E\es(\sum_{i=1}^{r_n}\mathds{1}_{\{\mb{X}_i\not\leq \mb{u}_n^{(\pmb{\tau})}\}}|\sum_{i=1}^{r_n}\mathds{1}_{\{\mb{X}_i\not\leq \mb{u}_n^{(\pmb{\tau})}\}}>0\di)\,,
\end{array}
\end{eqnarray}
for sequences $r_n=[n/k_n]$ and $k_n=o(n)$ provided that  $\{\mb{X}_n\}$ satisfies condition strong-mixing.\\

The multivariate extremal index thus preserves, with the natural adaptations, the characteristics that made famous the univariate extremal index. Additionally to these similar characteristics to the univariate extremal index, it plays an unavoidable role in the tail dependence characterization of $H$. If the tail dependence coefficients applied to $F$ remain unchanged when applied to $\widehat{H}$ (Li \cite{li09} 2009), we can not guarantee the same for $H$, as will be seen in Section \ref{sIEMtaildep}. The presence of serial dependence within each marginal sequence and between marginal sequences, makes it impossible to approximate the dependence coefficients in the tail of $\mb{M}_n$ to those of $F$.

The dependence modeling between the marginals of $F$ has received considerably more attention in literature than the dependence between the marginals of $F_{\mb{M}_n}$, which differs from $F_{\mb{\widehat{M}}_n}=F^n$ for being affected by $\theta(\pmb{\tau})$. The need to characterize this dependence appears, for instance, when we have a random field $\{\mb{X}_{\mb{i},n}, \mb{i}\in \ze^2,\, n\geq 1\}$ and we consider random vectors $(X_{i_1,n},\hdots,X_{i_s,n})$ corresponding to locations $(i_1,\hdots,i_s)$ at time instant $n$. The sequence $\{(X_{i_1,n},\hdots,X_{i_s,n})\}_{n\geq 1}$ presents in general a multivariate extremal index $\theta_{i_1,\hdots,i_s}(\pmb{\tau})$ encompassing information about dependence in the space of locations $i_1,\hdots,i_s$ and when the time $n$ varies (Ferreira \ea~\cite{fer+15} 2016). Relation (\ref{HeHchap}) applied to MEV distributions $\widehat{H}$ and functions $\theta(x_1,\hdots,x_d)$ compatible with the properties of a multivariate extremal index, provide a means of constructing MEV distributions (Martins and Ferreira \cite{mar+fer05} 2005).

Notwithstanding all these challenges posed by and for the multivariate extremal index, the literature proves that it remained on the theoretical shelves of the study of extreme values.

The main difficulty of applying the multivariate extremal index lies in the fact that it is a function, unlike what happens with the marginal univariate extremal indexes, for which we have several estimation methods in the literature (see, e.g., Gomes \ea~\cite{gom+08} 2008, Northrop \cite{nor15} 2015, Ferreira and Ferreira \cite{fer+fer16} 2016 and references therein).

Since it remains present the need to estimate the propensity for clustering in a context of multivariate sequences, we propose in this work: (a) decompose it, highlighting different types of information contained in it; (b) bound it in order to obtain a better upper limit than those available in the literature; (c)  enhance its role in the dependence of the tail of $H$; (d) apply it to models of recognized interest in applications.

\section{Co-movements point processes}\label{scomovpp}

Based on (\ref{IEMclustmean}) the \mei~can be seen as the number of the limiting mean dimension of clustering of events counted by the point process
\begin{eqnarray}\nn
N_n=\sum_{i=1}^n\mathds{1}_{\{\mb{X}_i\not\leq \mb{u}_n^{(\pmb{\tau})}\}}.
\end{eqnarray}

We are going to consider two point processes of more restricted events, corresponding to joint exceedances for various marginals of  $\mb{X}_i$ and enhance the contribution of the extremal indexes of these events in the value of $\theta(\pmb{\tau})$.

Let, for each $\emptyset\not=J\subset D=\{1,\hdots,d\}$,
\begin{eqnarray}\nn
N_{n,J}^{*}=\sum_{i=1}^n\mathds{1}_{\{\bigcap_{j\in J} \{X_{ij}>u_{nj}\}\}}, \, n\geq 1,
\end{eqnarray}
and
\begin{eqnarray}\nn
N_{n,J}^{**}=\sum_{i=1}^n\mathds{1}_{\{\bigwedge_{j\in J}X_{ij}>\bigvee_{j\in J}u_{nj}\}}, \, n\geq 1,
\end{eqnarray}
where notations $\wedge$ and $\vee$ stand for minimum and maximum, respectively.

We denote the respective limiting mean number of occurrences by
\begin{eqnarray}\nn
\Gamma_{J}^{*}(\pmb{\tau}_J)=\lim_{n\to\infty}n\dst P\es(\bigcap_{j\in J} \{X_{ij}>u_{nj}\}\di)
\end{eqnarray}
and
\begin{eqnarray}\nn
\Gamma_{J}^{**}(\pmb{\tau}_J)=\lim_{n\to\infty}n\dst P\es(\bigcap_{j\in J} \{X_{ij}>\bigvee_{j\in J}u_{nj}\}\di).
\end{eqnarray}
Observe that
\begin{eqnarray}\nn
\Gamma_{J}^{**}(\pmb{\tau}_J)=\lim_{n\to\infty}nP\es(\dst \bigcap_{j\in J} \es\{X_{ij}>\frac{n}{\bigwedge_{j\in J}\tau_{j}}\di\}\di).
\end{eqnarray}
Thus
\begin{eqnarray}\nn
\Gamma_{J}^{**}(\pmb{\tau}_J)=\tau^{**}_J\es(\bigwedge_{j\in J}\tau_{j}\di),
\end{eqnarray}
with $\tau^{**}_J$ an increasing function in $\bigwedge_{j\in J}\tau_{j}$ and homogeneous of order $1$. Therefore, we have
\begin{eqnarray}\label{reltau**}
\tau^{**}_J\es(\frac{\bigwedge_{j\in J}\tau_{j}}{s}\di)=\frac{\tau^{**}_J\es(\bigwedge_{j\in J}\tau_{j}\di)}{s},
\end{eqnarray}
for all $s\not=0$, a relation that will be fundamental for the independence of $\theta^{**}$ from $\tau$.

In case $J=D$, we will omit the index $J$ in notation.

For each of these processes, we can define an index of clustering of occurrences, which we will also call extremal indexes, $\theta^{*}_J\es(\pmb{\tau}_J\di)$ and $\theta^{**}_J\es(\pmb{\tau}_J\di)$, being the latter a constant independent of $\pmb{\tau}_J$, as we will see.

Let us assume that sequence $\{\mb{X}_n\}_{n\geq 1}$ satisfies the strong-mixing condition (Leadbetter \ea~\cite{lead+83} 1983) and, as consequence, we have, as $n\to\infty$,
\begin{eqnarray}\nn
P\es(N_{n,J}^{}=0\di)-P^{k_n}\es(N_{[n/k_n],J}^{}=0\di)\to 0,
\end{eqnarray}
\begin{eqnarray}\nn
P\es(N_{n,J}^{*}=0\di)-P^{k_n}\es(N_{[n/k_n],J}^{*}=0\di)\to 0
\end{eqnarray}
and
\begin{eqnarray}\nn
P\es(N_{n,J}^{**}=0\di)-P^{k_n}\es(N_{[n/k_n],J}^{**}=0\di)\to 0,
\end{eqnarray}
for any integers sequence $\{k_n\}$, such that, $k_n\to\infty$, $k_n\alpha_n(l_n)\to 0$ and $k_nl_n/n\to 0$, as $n\to\infty$, where $\alpha_n(\cdot)$ and $l_n$ are the sequences of the strong-mixing condition. Thus
\begin{eqnarray}\nn
P\es(N_{n,J}^{}=0\di)\to \exp\es(-\theta_J(\pmb{\tau}_J)\Gamma_J(\pmb{\tau}_J)\di),
\end{eqnarray}
\begin{eqnarray}\nn
P\es(N_{n,J}^{*}=0\di)\to \exp\es(-\theta_J^{*}(\pmb{\tau}_J)\Gamma_J^{*}(\pmb{\tau}_J)\di)
\end{eqnarray}
and
\begin{eqnarray}\label{N**lim}
P\es(N_{n,J}^{**}=0\di)\to \exp\es(-\theta_J^{**}(\pmb{\tau}_J)\tau^{**}_J\es(\bigwedge_{j\in J}\tau_{j}\di)\di),
\end{eqnarray}
with
\begin{eqnarray}\nn
\theta_J(\pmb{\tau}_J)=\lim_{n\to\infty}k_nP\es(N_{[n/k_n],J}^{}>0\di)/\Gamma_J(\pmb{\tau}_J),
\end{eqnarray}
\begin{eqnarray}\nn
\theta_J^{*}(\pmb{\tau}_J)=\lim_{n\to\infty}k_nP\es(N_{[n/k_n],J}^{*}>0\di)/\Gamma_J^{*}(\pmb{\tau}_J),
\end{eqnarray}
\begin{eqnarray}\nn
\theta_J^{**}(\pmb{\tau}_J)=\lim_{n\to\infty}k_nP\es(N_{[n/k_n],J}^{**}>0\di)/\tau^{**}_J\es(\bigwedge_{j\in J}\tau_{j}\di)
\end{eqnarray}
and
\begin{eqnarray}\nn
\theta_J^{**}(\pmb{\tau}_J)\tau^{**}_J\es(\bigwedge_{j\in J}\tau_{j}\di)\leq \theta_J^{*}(\pmb{\tau}_J)\Gamma_J^{*}(\pmb{\tau}_J)\leq \bigvee_{j\in J}\theta_j\tau_j\leq \theta_J(\pmb{\tau}_J)\Gamma_J(\pmb{\tau}_J).
\end{eqnarray}
In the following we present relations between $\theta_J^{**}(\pmb{\tau}_J)$, $\theta_J^{*}(\pmb{\tau}_J)$ and $\theta_J^{}(\pmb{\tau}_J)$, which will allow us a detailed interpretation of the information contained in $\theta(\pmb{\tau})$ and an upper bound better than the one in (\ref{IEMlims}). But first, we start by proving that $\theta^{**}_J(\pmb{\tau}_J)=\theta^{**}_J$, i.e., these extremal indexes are independent of $\pmb{\tau}$, which is already known for $J=\{j\}$ (Leadbetter \ea~\cite{lead+83} 1983), $j=1,\hdots,d$, since $\theta^{**}_{\{j\}}=\theta_j$. Indeed the proof runs along the same lines.

\begin{pro}\label{p1}
For stationary sequences $\{\mb{X}_n\}$ satisfying the strong-mixing condition, if there exists the limit (\ref{N**lim}) for some $\pmb{\tau}$, then it exists for any $\pmb{\tau}>0$ and we have
\begin{eqnarray}\nn
P\es(N_{n,A}^{**}=0\di)\to \exp\es(-\theta_A^{**}\tau^{**}_A\es(\bigwedge_{j\in A}\tau_{j}\di)\di),
\end{eqnarray}
with $\theta^{**}_A\in[0,1]$ constant.
\end{pro}
\begin{proof}
From the strong-mixing condition, we have
\begin{eqnarray}\nn
\begin{array}{rl}
&\dst\liminf_{n\to\infty} P\es(N_{n,A}^{**}=0\di)= \dst\liminf_{n\to\infty} P^{k_n}\es(N_{[n/k_n],A}^{**}=0\di)= \dst\liminf_{n\to\infty}\es(1-\frac{k_n P\es(N_{[n/k_n],A}^{**}>0\di)}{k_n}\di)^{k_n}\\\\
\geq & \dst\liminf_{n\to\infty}\es(1-\frac{n P\es(\bigwedge_{j\in A} X_{1j}>\bigvee_{j\in A}u_{nj}\di)}{k_n}\di)^{k_n}= \dst \es(1-\frac{\tau^{**}_A\es(\bigwedge_{j\in A} \di)}{k_n}\di)^{k_n}.
\end{array}
\end{eqnarray}
Thus, if there exists $\Psi(\tau^{**}_A)=\limsup_{n\to\infty}P\es(N_{n,A}^{**}=0\di)$, we have $\Psi\es(\tau^{**}_A\es(\bigwedge_{j\in A} \di)\di)\geq \exp\es(-\tau^{**}_A\es(\bigwedge_{j\in A} \di)\di)$, and so $\Psi(\tau^{**}_A)$ is a strictly positive function.

We also have that function $\Psi(\tau^{**}_A)$ would have to satisfy  $\Psi(\tau^{**}_A/k)=\Psi^{1/k}(\tau^{**}_A)$, for all $\tau^{**}_A>0$ and $k=1,2,\hdots$, since, representing $\sum_{i=1}^n\mathds{1}_{\{\bigwedge_{j\in A} X_{ij}>m/\bigwedge_{j\in A}\tau_j\}}$ by $N_n^{**}\es(\mb{u}_m^{\es(\tau^{**}_A(\bigwedge_{j\in A} \tau_j)\di)}\di)$ and applying (\ref{N**lim}), it holds
\begin{eqnarray}\nn
\begin{array}{rl}
&\dst \es|P\es(N_{[n/k_n],A}^{**}\es(\mb{u}_n^{\es(\tau^{**}_A(\bigwedge_{j\in A} \tau_j)\di)}\di)=0\di)
-P\es(N_{[n/k_n],A}^{**}\es(\mb{u}_{[n/k_n]}^{\es(\tau^{**}_A(\bigwedge_{j\in A} \tau_j)/k_n\di)}\di)=0\di)\di|\\\\
\leq & \dst \es[\frac{n}{k_n}\di]\es|P\es(\bigwedge_{j\in A}X_{1j}>\frac{n}{\bigwedge_{j\in A}\tau_j}\di)
-P\es(\bigwedge_{j\in A}X_{1j}>\frac{[n/k_n]}{\bigwedge_{j\in A}\tau_j/k_n}\di)\di|\\\\
= &\dst \es[\frac{n}{k_n}\di]\es|\frac{\bigwedge_{j\in A}\tau_j}{n}(1+o(1))
- \frac{\bigwedge_{j\in A}\tau_j/k_n}{[n/k_n]}(1+o(1))\di|=o(1)
\end{array}
\end{eqnarray}
and thus we would have
\begin{eqnarray}\nn
\begin{array}{rl}
&\dst \Psi\es(\frac{\tau^{**}_A}{k_n}\di)
=\limsup_{n\to\infty}P\es(N^{**}_{[n/k_n],A}\es(\mb{u}_{[n/k_n],A}^{(\tau^{**}_A/k_n)}\di)=0\di)
=\limsup_{n\to\infty}P\es(N^{**}_{n,A}\es(\mb{u}_{n,A}^{(\tau^{**}_A)}\di)=0\di)
=\Psi\es(\tau^{**}_A\di)^{1/k_n}.
\end{array}
\end{eqnarray}
On the other hand, $\Psi\es(\tau^{**}_A\di)$ would have to be a non increasing function because if
\begin{eqnarray}\nn
\begin{array}{rl}
&\dst\tau^{**}_{0,A}\es(\bigwedge_{j\in A}\tau_{0,j}\di)
=\lim_{n\to\infty}n P\es(\bigwedge_{j\in A} X_{1j}>\frac{n}{\bigwedge_{j\in A}\tau_{0,j}}\di)
>\tau^{**}_{A}\es(\bigwedge_{j\in A}\tau_{j}\di)=\lim_{n\to\infty}n P\es(\bigwedge_{j\in A} X_{1j}>\frac{n}{\bigwedge_{j\in A}\tau_{j}}\di)
\end{array}
\end{eqnarray}
and $\tau^{**}_{A}\es(\bigwedge_{j\in A}\tau_{j}\di)$ is increasing in $\bigwedge_{j\in A}\tau_{j}$, then from some order,
\begin{eqnarray}\nn
\begin{array}{rl}
&\dst\es\{\bigwedge_{j\in A}X_{1j}>\frac{n}{\bigwedge_{j\in A}\tau_j}\di\}\subset
\es\{\bigwedge_{j\in A}X_{1j}>\frac{n}{\bigwedge_{j\in A}\tau_{0,j}}\di\}
\end{array}
\end{eqnarray}
and thus
\begin{eqnarray}\nn
\begin{array}{rl}
&\dst\es\{N^{**}_{n,A}\es(\mb{u}_n^{\es(\tau^{**}_{0,A}\di)}\di)=0\di\}\subset
\es\{N^{**}_{n,A}\es(\mb{u}_n^{\es(\tau^{**}_{A}\di)}\di)=0\di\}
\end{array}
\end{eqnarray}
and $\Psi\es(\tau^{**}_{0,A}\di)\leq \Psi\es(\tau^{**}_{A}\di)$. If $\Psi\es(\tau^{**}_{A}\di)$ is a strictly positive function, non increasing and such that $\Psi\es(\tau^{**}_{A}/k\di)=\Psi\es(\tau^{**}_{A}\di)^{1/k}$, then $\Psi\es(\tau^{**}_{A}\di)=\exp\es(-\theta^{**}_A\tau^{**}_A\di)$, with $\theta^{**}_A$ a non negative constant. Since $\Psi\es(\tau^{**}_{A}\di)>\exp\es(-\tau^{**}_A\di)$, it also comes $\theta^{**}_A\leq 1$. For the lower limit, we can make the same reasoning to obtain the result.
\end{proof}

Let us start by emphasizing that, to $\theta(\pmb{\tau})\Gamma(\pmb{\tau})$, we have the contribution of the clustering of the joint exceedances of all levels by the respective marginals, including the particular case of the clustering of exceedances of the largest level by the lower marginal, as well as, the clustering of exceedances of one or more levels by the respective marginals without joint exceedances of all levels.\\

\begin{pro}\label{p2}
Let $\{\mb{X}_n\}$ be a stationary sequence satisfying the strong-mixing condition and $\{\mb{u}_n^{(\pmb{\tau})}=(u_{n}^{(\tau_1)},\hdots,u_{n}^{(\tau_d)})\}$ a sequence of normalized real levels for which there exists $\Gamma(\pmb{\tau})$. Then
$$\theta(\pmb{\tau})\Gamma(\pmb{\tau})=\theta^{**}\tau^{**}\es(\bigwedge_{j=1}^d\tau_j\di)
    +\theta^{*}(\pmb{\tau})\Gamma^{*}(\pmb{\tau})\beta ^{(1)}(\pmb{\tau})+\sum_{\emptyset\not=J\subset D}(-1)^{|J|+1}\Theta_J(\pmb{\tau}_J),$$
    where
    $$\beta ^{(1)}(\pmb{\tau})=\lim_{n\to\infty}P\es(N^{**}_{r_n}=0|N^{*}_{r_n}>0\di)$$ and

    $$\Theta_J(\pmb{\tau}_J)=\lim_{n\to\infty}k_nP\es(\bigcap_{j\in J}\{N^{}_{r_n,\{j\}}>0\}|N^{*}_{r_n}=0\di).$$

\end{pro}
\begin{proof}
We have
\begin{eqnarray}\nn
\begin{array}{rl}
k_nP\es(N_{r_n}>0\di)=&\dst k_nP\es(N^{**}_{r_n}>0\di)+k_nP\es(N^{*}_{r_n}>0,N^{**}_{r_n}=0\di)
+k_nP\es(N^{}_{r_n}>0,N^{*}_{r_n}=0\di)\\\\
=& \dst k_nP\es(N^{**}_{r_n}>0\di)+k_nP\es(N^{*}_{r_n}>0\di)P\es(N^{**}_{r_n}=0|N^{*}_{r_n}>0\di)\\
&+k_nP\es(\bigcup_{j=1}^d\{N^{}_{r_n,\{j\}}>0\},N^{*}_{r_n}=0\di)\\\\
=&\dst \dst k_nP\es(N^{**}_{r_n}>0\di)+k_nP\es(N^{*}_{r_n}>0\di)P\es(N^{**}_{r_n}=0|N^{*}_{r_n}>0\di)\\
&+\dst\sum_{\emptyset\not= J\subset D}(-1)^{|J|+1}k_nP\es(\bigcap_{j\in J}\{N^{}_{r_n,\{j\}}>0\},N^{*}_{r_n}=0\di).
\end{array}
\end{eqnarray}
In what concerns the last term, observe that
\begin{eqnarray}\nn
\dst\sum_{j=1}^d k_nP\es(N^{}_{r_n,\{j\}}>0,N^{*}_{r_n}=0\di)=\sum_{j=1}^d k_nP\es(N^{}_{r_n,\{j\}}>0\di)-\sum_{j=1}^d k_nP\es(N^{}_{r_n,\{j\}}>0,N^{*}_{r_n}>0\di)
\end{eqnarray}
and since $\lim_{n\to\infty}P\es(N^{*}_{r_n}=0\di)=1$, we have the result.
\end{proof}

Observe that $\beta ^{(1)}(\pmb{\tau})$ reduces  $\theta^{*}(\pmb{\tau})$ from the joint exceedances of $\bigvee_{j=1}^d n/\tau_j$ accounted for $\theta^{**}$. We can say that in the last term of representation of $\theta(\pmb{\tau})\Gamma(\pmb{\tau})$ we are accounting the clustering propensity concerning one or more marginals, without joint exceedances of all the marginals.

We illustrate the previous result with a bivariate sequence  with unit Fr\'{e}chet marginals and such that the joint tail is regularly varying at $\infty$ with index $\eta\in(0,1]$ measuring a penultimate tail dependence, as the (sub)model presented in Ledford and Tawn (\cite{led+taw96,led+taw97} 1996,1997).

\begin{ex}\label{ex1}
Suppose that $d=2$ and $\{(X_{n1},X_{n2})\}_{n\geq 1}$ is a strong-mixing stationary sequence, with unit Fr\'{e}chet marginals and such that
\begin{eqnarray}\label{etaex1}
P\es(X_{n1}>x,X_{n2}>x\di)\sim x^{-1/\eta}L(x),
\end{eqnarray}
as $x\to\infty$, where $0<\eta<1$ and $L$ is a slowly varying function, i.e., $L(tx)/L(x)\to 1$, $\forall t>0$. Then
\begin{eqnarray}\nn
\theta^{**}=k_nP\es(N^{**}_{r_n}>0\di)\leq nP\es(X_{n1}>\frac{n}{\tau_1\wedge\tau_2},X_{n2}>\frac{n}{\tau_1\wedge\tau_2}\di)\sim n\es(\frac{n}{\tau_1\wedge\tau_2}\di)^{-1/\eta}L\es(\frac{n}{\tau_1\wedge\tau_2}\di)\to 0,
\end{eqnarray}
\begin{eqnarray}\nn
\theta^{*}(\tau_1,\tau_2)\leq nP\es(X_{n1}>\frac{n}{\tau_1},X_{n2}>\frac{n}{\tau_2}\di)\leq nP\es(X_{n1}>\frac{n}{\tau_1\wedge\tau_2},X_{n2}>\frac{n}{\tau_1\wedge\tau_2}\di)\to 0.
\end{eqnarray}
Therefore, regardless of additional conditions on the serial dependence, the validity of (\ref{etaex1}) implies
\begin{eqnarray}\nn
\theta(\pmb{\tau})\Gamma(\pmb{\tau})=\sum_{\emptyset\not=J\subset \{1,2\}}(-1)^{|J|+1}\lim_{n\to\infty}k_nP\es(\bigcap_{j\in J}\{N_{r_n,\{j\}}>0\},N^{*}_{r_n}=0\di)
\end{eqnarray}
and $\Gamma(\pmb{\tau})=\tau_1+\tau_2$. Since $k_nP\es(N^{*}_{r_n}>0\di)\to 0$ we can thus write in this model
\begin{eqnarray}\label{ex1teta1}
\theta(\pmb{\tau})=\frac{1}{\tau_1+\tau_2}\lim_{n\to\infty}k_n\es(P\es(N_{r_n,\{1\}}>0\di)
+P\es(N_{r_n,\{2\}}>0\di)
-P\es(N_{r_n,\{1\}}>0,N_{r_n,\{2\}}>0\di)\di).
\end{eqnarray}

We now consider several particular situations.

(a) In the case of independent vectors $(X_{n1},X_{n2})$, $n\geq 1$, we have
\begin{eqnarray}\nn
\begin{array}{rl}
\theta(\pmb{\tau})=&\dst \frac{1}{\tau_1+\tau_2}\es(\theta_1\tau_1+\theta_2\tau_2-\lim_{n\to\infty}k_nP\es(\bigcup_{1\leq i<i'\leq r_n}\es\{\es\{X_{i1}>\frac{n}{\tau_1},X_{i2}\leq\frac{n}{\tau_2},X_{i'1}\leq\frac{n}{\tau_1}
,X_{i'2}>\frac{n}{\tau_2},\di\}\di.\di.\di.\\\\
&\dst \es.\es.\es.\bigcup
\es\{X_{i1}\leq\frac{n}{\tau_1},X_{i2}>\frac{n}{\tau_2},X_{i'1}>\frac{n}{\tau_1}
,X_{i'2}\leq\frac{n}{\tau_2}\di\}\di\}\di)\di)=\frac{\tau_1+\tau_2}{\tau_1+\tau_2}=1.
\end{array}
\end{eqnarray}
It will then come $P\es(M_{n1}\leq n/\tau_1,M_{n2}\leq n/\tau_2\di)\to \exp(-\Gamma(\pmb{\tau}))=\exp(-\tau_1)\exp(-\tau_2)$, that is, $M_{n1}$ and $M_{n2}$ are also asymptotically independent. \\

(b) Suppose that  $\{(X_{n1},X_{n2})\}_{n\geq 1}$, satisfies condition $D^{(m)}_{\{1,2\}}$ defined by
\begin{eqnarray}\nn
\lim_{n\to\infty}n\sum_{j=m+1}^{[n/k_n]}P\es(X_{11}>n/\tau_1,X_{j2}>n/\tau_2\di)=0,
\end{eqnarray}
which extends $D^{'}_{\{1,2\}}$ of Davis (\cite{dav82} 1982), satisfied by i.i.d.~sequences. Then
\begin{eqnarray}\nn
\begin{array}{rl}
\theta(\pmb{\tau})=&\dst \frac{1}{\tau_1+\tau_2}\es(\theta_1\tau_1+\theta_2\tau_2
-\lim_{n\to\infty}n\sum_{i=2}^{m}P\es(X_{11}>n/\tau_1,X_{i2}>n/\tau_2\di)\di),
\end{array}
\end{eqnarray}
where the last part reflects the cross dependence. \\

(c) If we assume an analogous hypothesis of (\ref{etaex1}) for $(X_{11},X_{i2})$ with different $\eta_i$, we will also obtain asymptotic independence between $M_{n1}$ and $M_{n2}$, since the last term has null limit. We have $P\es(M_{n1}\leq n/\tau_1,M_{n2}\leq n/\tau_2\di)\to \exp(-\Gamma(\pmb{\tau})\theta(\pmb{\tau}))=\exp(-\theta_1\tau_1)\exp(-\theta_2\tau_2)$.\\

(d) If $\theta(\pmb{\tau})=\theta$, $\forall \pmb{\tau}\in \er^2_+$, then $\theta_1=\theta_2=\theta$ and, from (\ref{ex1teta1}),
\begin{eqnarray}\nn
\theta=\theta-\lim_{n\to\infty}k_nP\es(N_{r_n,\{1\}}>0,N_{r_n,\{2\}}>0\di),
\end{eqnarray}
which implies that this limit is null and thus $P\es(M_{n1}\leq n/\tau_1,M_{n2}\leq n/\tau_2\di)\to \exp(-\theta(\tau_1+\tau_2))=\exp(-\theta\tau_1)\exp(-\theta\tau_2)$.
\end{ex}

We present below a relation between $\theta(\pmb{\tau})$ and the extremal indexes $\theta^{**}_{\{j,\hdots,d\}}$ and $\theta^{*}_{\{j,\hdots,d\}}\es(\pmb{\tau}_{\{j,\hdots,d\}}\di)$, $j=1,\hdots,d$, which discriminates different informations contained in function  $\theta(\pmb{\tau})$ and provides an upper bound for $\theta(\pmb{\tau})$ better than the one in (\ref{IEMlims}). In Example \ref{ex2} we show that the proposed upper bound for the M4 processes, can be better than the one presented in Ehlert and Schlather (\cite{ehl+Sch08} 2008). The new upper bound has also the advantage of depending only on constant extremal indexes which can be estimated by known methods of literature.

\begin{pro}\label{p3}
Let $\{\mb{X}_n\}$ be a stationary sequence satisfying the strong-mixing condition and $\{\mb{u}_n^{(\pmb{\tau})}=(u_{n}^{(\tau_1)},\hdots,u_{n}^{(\tau_d)})\}$ a sequence of normalized real levels for which there exists $\Gamma(\pmb{\tau})$. Then
\begin{enumerate}
\item[(a)]
\begin{eqnarray}\nn
\begin{array}{rl}
\theta(\pmb{\tau})\Gamma(\pmb{\tau})=&\dst\lim_{n\to\infty}k_nP\es(N^{}_{r_n}>0\di)
=\sum_{j=1}^d\theta_j\tau_j-
\sum_{j=1}^{d-1}\theta^{**}_{\{j,\hdots,d\}}\tau^{**}_{\{j,\hdots,d\}}\es(\bigwedge_{i=j}^d\tau_i\di)\\\\
&-\dst
\sum_{j=1}^{d-1}\theta^{*}_{\{j,\hdots,d\}}\es(\pmb{\tau}^{}_{\{j,\hdots,d\}}\di)
\Gamma^{*}_{\{j,\hdots,d\}}\es(\pmb{\tau}^{}_{\{j,\hdots,d\}}\di)
\beta^{(1)}_j\es(\pmb{\tau}^{}_{\{j,\hdots,d\}}\di)\\\\
& \dst -\sum_{j=1}^{d-1}\sum_{J\subset \{j+1.\hdots,d\}}(-1)^{|J|+1}\beta^{(2)}_{\{j\}\cup J}\es(\pmb{\tau}_{\{j\}\cup J}\di),
\end{array}
\end{eqnarray}
where $\beta^{(1)}_j\es(\pmb{\tau}^{}_{\{j,\hdots,d\}}\di)=
\lim_{n\to\infty}P\es(N^{**}_{r_n,\{j,\hdots,d\}}=0|N^{*}_{r_n,\{j,\hdots,d\}}>0\di)$ and
$\beta^{(2)}_{\{j\}\cup J}\es(\pmb{\tau}_{\{j\}\cup J}\di)=
\lim_{n\to\infty}k_nP\es(\bigcap_{i\in\{j\}\cup J}\{N^{}_{r_n,\{i\}}>0\}|N^{*}_{r_n,\{j,\hdots,d\}}=0\di)$, provided that the limiting constants exist.

\item[(b)] $\theta(\pmb{\tau})\leq\frac{1}{\Gamma(\pmb{\tau})}\es(\sum_{j=1}^d\theta_j\tau_j
-\sum_{j=1}^{d-1}\theta^{**}_{\{j,\hdots,d\}}\tau^{**}_{\{j,\hdots,d\}}\es(\bigwedge_{i=j}^d\tau_i\di)\di).$\\
\end{enumerate}
\end{pro}
\begin{proof}
We have
\begin{eqnarray}\nn
\begin{array}{rl}
k_nP\es(N_{r_n}>0\di)=&\dst k_nP\es(\bigcup_{j=1}^d\{N^{}_{r_n,\{j\}}>0\}\di)\\\\
=&\dst \sum_{j=1}^{d-1}k_n P\es(N^{}_{r_n,\{j\}}>0,\bigcap_{i=j+1}^d\{N^{}_{r_n,\{i\}}=0\}\di)
+k_nP\es(N^{}_{r_n,\{d\}}>0\di)\\\\
=&\dst \sum_{j=1}^{d}k_n P\es(N^{}_{r_n,\{j\}}>0\di)
-\sum_{j=1}^{d-1}k_n P\es(N^{}_{r_n,\{j\}}>0,\bigcup_{i=j+1}^d\{N^{}_{r_n,\{i\}}>0\}\di).
\end{array}
\end{eqnarray}
Regarding the second term, we can also say that
\begin{eqnarray}\nn
\begin{array}{rl}
&\dst \sum_{j=1}^{d-1}k_n P\es(N^{}_{r_n,\{j\}}>0,\bigcup_{i=j+1}^d\{N^{}_{r_n,\{i\}}>0\}\di)\\\\
=& \dst \sum_{j=1}^{d-1}k_n P\es(N^{}_{r_n,\{j\}}>0,\bigcup_{i=j+1}^d\{N^{}_{r_n,\{i\}}>0\},N^{*}_{r_n,\{j,\hdots,d\}}>0\di)\\\\
&\dst + \sum_{j=1}^{d-1}k_n P\es(N^{}_{r_n,\{j\}}>0,\bigcup_{i=j+1}^d\{N^{}_{r_n,\{i\}}>0\},N^{*}_{r_n,\{j,\hdots,d\}}=0\di)\\\\
=& \dst \sum_{j=1}^{d-1}k_n P\es(N^{*}_{r_n,\{j,\hdots,d\}}>0,N^{**}_{r_n,\{j,\hdots,d\}}>0\di)\\\\
& \dst + \sum_{j=1}^{d-1}k_n P\es(N^{*}_{r_n,\{j,\hdots,d\}}>0,N^{**}_{r_n,\{j,\hdots,d\}}=0\di)\\\\
& \dst + \sum_{j=1}^{d-1}k_n P\es(N^{}_{r_n,\{j\}}>0,\bigcup_{i=j+1}^d\{N^{}_{r_n,\{i\}}>0\},N^{*}_{r_n,\{j,\hdots,d\}}=0\di)\\\\
=& \dst \sum_{j=1}^{d-1}k_n P\es(N^{**}_{r_n,\{j,\hdots,d\}}>0\di)\\\\
& \dst + \sum_{j=1}^{d-1}k_n P\es(N^{*}_{r_n,\{j,\hdots,d\}}>0,N^{**}_{r_n,\{j,\hdots,d\}}=0\di)\\\\
& \dst + \sum_{j=1}^{d-1}k_n P\es(N^{}_{r_n,\{j\}}>0,\bigcup_{i=j+1}^d\{N^{}_{r_n,\{i\}}>0\},N^{*}_{r_n,\{j,\hdots,d\}}=0\di).
\end{array}
\end{eqnarray}
Therefore,
\begin{eqnarray}\nn
\begin{array}{rl}
\theta(\pmb{\tau})\Gamma(\pmb{\tau})=&\dst\lim_{n\to\infty}k_nP\es(N^{}_{r_n}>0\di)
=\sum_{j=1}^d\theta_j\tau_j-
\sum_{j=1}^{d-1}\theta^{**}_{\{j,\hdots,d\}}\tau^{**}_{\{j,\hdots,d\}}\es(\bigwedge_{i=j}^d\tau_i\di)\\\\
&-\dst
\sum_{j=1}^{d-1}\theta^{*}_{\{j,\hdots,d\}}\es(\pmb{\tau}^{}_{\{j,\hdots,d\}}\di)
\Gamma^{*}_{\{j,\hdots,d\}}\es(\pmb{\tau}^{}_{\{j,\hdots,d\}}\di)
\lim_{n\to\infty}P\es(N^{**}_{r_n,\{j,\hdots,d\}}=0|N^{*}_{r_n,\{j,\hdots,d\}}>0\di)\\\\
& \dst -\sum_{j=1}^{d-1}\lim_{n\to\infty}k_nP\es(\bigcup_{i=j+1}^d\{N^{}_{r_n,\{j\}}>0,N^{}_{r_n,\{i\}}>0\}
|N^{*}_{r_n,\{j,\hdots,d\}}=0\di),
\end{array}
\end{eqnarray}
since $P\es(N^{*}_{r_n,\{j,\hdots,d\}}=0\di)\to 1$, as $n\to\infty$.
\end{proof}

The above result means that, for each $j\in \{1,\hdots,d\}$, the values
$\theta^{*}_{\{j,\hdots,d\}}\es(\pmb{\tau}^{}_{\{j,\hdots,d\}}\di)$ only contribute to $\theta(\pmb{\tau})$ if it is not asymptotically almost surely the local occurrence of some joint exceedances of the largest level $u_{ni}^{(\tau_j)}$, $i\in \{1,\hdots,d\}$, among the joint exceedances of these levels. Otherwise, the joint exceedances clustering is considered only through the clustering of the joint exceedances of the largest level $u_{ni}$, $i\in \{j,\hdots,d\}$, and measured by $\theta^{**}_{\{j,\hdots,d\}}$, disappearing the third term. Therefore, the second and third terms together account for the clustering of two situations of joint exceedances. The fourth term measures the clustering of exceedances of $u_{nj}$ and of one or more $u_{ni}$, $i\in \{j+1,\hdots,d\}$, in the absence of joint exceedances of levels $u_{ni}$, $i\in \{j,\hdots,d\}$, not accounted within the second and third terms. All these clustering situations were accounted by excess in the first term.\\

The function $\theta(\pmb{\tau})$ is homogeneous of order zero and thus $\theta(\tau,\hdots,\tau)=\theta(1,\hdots,1)$, $\forall \tau \in\er$. The constant $\theta(1,\hdots,1)$ has been used as a dependence coefficient of the marginals of $H$ (see, e.g., Martins and Ferreira \cite{mar+fer05} 2005, Ehlert and Schlather \cite{ehl+Sch08} 2008, Ferreira and Ferreira \cite{fer+fer15} 2015, and references therein).

We are going to analyze the consequences of the decompositions presented for $\theta(\pmb{\tau})$ in the calculation of $\theta(\mb{1})$.

If $\tau_1=\hdots=\tau_d=\tau$, then $N^{**}_{n}=N^{*}_{n}$, $\beta^{(1)}_{J}(\pmb{\tau})=0$, $\Gamma^{*}(\pmb{\tau})=\tau^{**}(\pmb{\tau})$ and $\Gamma^{}(\pmb{\tau})=\sum_{\emptyset\not=J\subset D}(-1)^{|J|+1}\tau^{**}_{J}(\pmb{\tau}_{J})$.

The first decomposition
\begin{eqnarray}\nn
\theta(\mb{1})\Gamma(\mb{1})=\theta^{**}\tau^{**}(\mb{1})
+\lim_{n\to\infty}k_nP\es(\bigcup_{j=1}^d\{N_{r_n,\{j\}}>0\},N^{*}_{r_n}=0\di),
\end{eqnarray}
separates once again the contribution of the clustering of exceedances across all marginals from the contribution of the clustering of exceedances of one or more marginals without exceedances of all marginals.

In the next section, we will give an important utility to the boundary of $\theta(\pmb{\tau})\Gamma^{*}(\pmb{\tau})$. It will serve to delimitate the difference between the tail dependence coefficients of $H$ and $\widehat{H}$.

The second decomposition allow us to obtain an upper bound for $\theta(\mb{1})$, which can be better than the one presented in (\ref{IEMlims}). From the previous result, we have
\begin{eqnarray}\label{IEM1tau1}
\theta(\mb{1})\Gamma(\mb{1})\leq\sum_{j=1}^d\theta_j-\sum_{j=1}^{d-1} \theta^{**}_{\{j,\hdots,d\}}\tau^{**}_{\{j,\hdots,d\}}(\mb{1}).
\end{eqnarray}
From the proof of Proposition \ref{p3} we found that, instead of following the order $1,\hdots,d$ to decompose initially the event $\{\bigcup_{j=1}^dN_{r_n,\{j\}}>0\}$ in a reunion of disjoint events $\{N_{r_n,\{j\}}>0,\bigcap_{i=j+1}^d\{N_{r_n,\{i\}}>0\}\}$, $j=1,\hdots,d-1$ and $\{N_{r_n,\{d\}}>0\}\}$, we can consider any other permutation $(i_1,\hdots,i_d)$ from $(1,\hdots,d)$ and repeat the process.
Therefore the previous upper limit can be improved in the following sense:
\begin{eqnarray}\nn
\theta(\mb{1})\Gamma(\mb{1})\leq\sum_{j=1}^d\theta_j-\bigvee_{(i_1,\hdots,i_d)\in\mathcal{P}_d} \sum_{j=i_1}^{i_{d-1}} \theta^{**}_{\{j,\hdots,i_d\}}\tau^{**}_{\{j,\hdots,i_d\}}(\mb{1}),
\end{eqnarray}
where $\mathcal{P}_d$ denotes the set of all permutations of $(1,\hdots,d)$.

\begin{ex}\label{ex2}
Consider the M4 process,
\begin{eqnarray}\nn
\es\{
\begin{array}{l}
X_{n1}=0.7 Z_n\vee 0.3 Z_{n-2}\\
X_{n2}=0.7 Z_{n-1}\vee 0.1 Z_{n-2}\vee 0.5 Z_{n-3},
\end{array}
\di.
\end{eqnarray}
with $\{Z_n\equiv Z_{1,n}\}$, where $\{Z_{l,n}\}$, $l\geq 1,\, n\geq 1$, is an array of independent unit Fr\'{e}chet random variables. We have $\theta_1=0.7$, $\theta_2=0.5$ and $\theta(\mb{1})\Gamma(\mb{1})=0.7$. Since $\{X_n\}_{n\geq 1}$ is $4$-independent, representing $\{X_{i1}>n/\tau,X_{i2}>n/\tau\}$ by $A_{i,n}$ and $\tau_1\wedge \tau_2=\tau$, we have that
\begin{eqnarray}\nn
\begin{array}{rl}
\theta^{**}_{\{1,2\}}\tau^{**}_{\{1,2\}}(\tau)=&\dst
\lim_{n\to\infty}nP\es(A_{3,n}\cap \overline{A}_{4,n}\cap \overline{A}_{5,n} \cap \overline{A}_{6,n}\di)\\\\
=&\dst
\lim_{n\to\infty}nP\es(\{0.1Z_1>n/\tau\}\cap \overline{A}_{4,n}\cap \overline{A}_{5,n} \cap \overline{A}_{6,n}\di)=\lim_{n\to\infty}nP\es(\{0.1Z_1>n/\tau\}\cap \overline{A}_{4,n}\di) \\\\
=& \dst \lim_{n\to\infty}nP\es(\es\{0.1Z_1>n/\tau, 0.5Z_1\leq n/\tau\di\}\cup \es\{0.1Z_1>n/\tau, 0.5Z_1> n/\tau\di\}\di)\\\\
=&\dst 0.1\tau=0.1(\tau_1\wedge \tau_2).
\end{array}
\end{eqnarray}
Therefore, Proposition \ref{p3} indicates that  $\theta(\mb{1})\Gamma(\mb{1})\leq 0.7+0.5-0.1=1.1$. The upper limit in this type of processes has no great interest since we have the theoretical expression for $\theta(\pmb{\tau})$. However, this example serves to show that our upper bound can be better than the one presented in Ehlert and Schlather (\cite{ehl+Sch08} 2008) for M4 processes. Indeed, by applying their Corollary 3, we obtain
\begin{eqnarray}\nn
\begin{array}{rl}
\theta(\mb{1})\Gamma(\mb{1})\leq &\dst\es(\Gamma(\mb{1})-\bigvee_{j=1}^2(1-\theta_j)\di)\wedge \sum_{j=1}^d\theta_j=\es(\es(0.7+0.4+0.3+0.5\di)-\es(0.3\vee 0.5\di)\di)\wedge 1.2\\\\
=& 1.4\wedge 1.2=1.2.
\end{array}
\end{eqnarray}

In the cases where the number of signatures of an M4 process exceeds the number of marginals, examples are easily constructed in which the Ehlert and Schlather upper limit is reduced to $\sum_{j=1}^d\theta_j$, being in these cases the lower limit of (\ref{IEM1tau1}) below this. Our upper bound still has the advantage of being applied to processes outside the max-stable class.
\end{ex}

\section{Effect of the extremal index in the tail of a bivariate extreme values distribution}\label{sIEMtaildep}

For each pair $(j,j')$, $j<j'$ belonging to $D$, consider the bivariate (upper) tail dependence coefficient $\chi^{F}_{jj'}\in[0,1]$ for random pair $(X_{nj},X_{nj'})$ with df $F_{jj'}$, discussed in Sibuya (\cite{sib60} 1960) and Joe (\cite{joe97} 1997), defined by
\begin{eqnarray}\nn
\dst \chi^{F}_{jj'}=\lim_{u\uparrow 1^+}P\es(F_j(X_{ij})>u|F_{j'}(X_{ij'})>u\di)
\end{eqnarray}
and coefficient $\overline{\chi}_{jj'}^F\in[-1,1]$ of Coles \ea~(\cite{col+99} 1999), defined by
\begin{eqnarray}\nn
\dst \overline{\chi}^{F}_{jj'}=\lim_{u\uparrow 1^+}\frac{2\log P\es(F_{j'}(X_{ij'})>u\di)}{\log P\es(F_j(X_{ij})>u,F_{j'}(X_{ij'})>u\di)}-1.
\end{eqnarray}

We can say that $\chi^{F}_{jj'}$ corresponds to the probability of one variable being high given that the other is high too. The case $\chi^{F}_{jj'}>0$ means asymptotic dependence between $X_{nj}$ and $X_{nj'}$ and whenever $\chi^{F}_{jj'}=0$ the variables are said to be asymptotically independent. Assuming $\chi^{F}_{jj'}>0$ within asymptotically independent data may carry to an over-estimation of
probabilities of extreme joint events (see, e.g., Ledford and Tawn \cite{led+taw96,led+taw97} 1996, 1997). Asymptotically independent models, i.e., having $\chi^{F}_{jj'}=0$, may exhibit a residual tail dependence rendering different degrees of dependence  at finite levels. Coefficient $\overline{\chi}^{F}_{jj'}$ is a suitable tail measure within this class. Thus the pair $(\chi^{F}_{jj'},\overline{\chi}^{F}_{jj'})$ is a useful tool in characterizing the extremal dependence: under asymptotic dependence we have $\overline{\chi}^{F}_{jj'}=1$ and $0<\chi^{F}_{jj'}\leq 1$ quantifies the strength of dependence between the variables  $(X_{nj},X_{nj'})$ and, within the class of asymptotic independence, we have $\chi^{F}_{jj'}=0$ and $-1\leq \overline{\chi}^{F}_{jj'}< 1$ measures the strength of dependence of the random pair.\\

Observe that, both measures can be calculated from the copula $C_{_{F_{jj'}}}(u,u)=F_{jj'}(F_j^{-1}(u),F_{j'}^{-1}(u))$, with
\begin{eqnarray}\nn
\dst \chi^{F}_{jj'}=2-\lim_{u\uparrow 1^+}\frac{\log C_{_{F_{jj'}}}(u,u)}{\log u}
\end{eqnarray}
and
\begin{eqnarray}\nn
\dst \overline{\chi}^{F}_{jj'}=\lim_{u\uparrow 1^+}\frac{2\log (1-u)}{\log \es(1-2u+C_{_{F_{jj'}}}(u,u)\di)}-1.
\end{eqnarray}

If $F$ belongs to the domain of attraction of $\widehat{H}$, then $\chi^{F}_{jj'}=\chi^{\widehat{H}}_{jj'}$ and $\overline{\chi}^{F}_{jj'}=\overline{\chi}^{\widehat{H}}_{jj'}$. This results from the uniform convergence of $C_{F}^n$ to $C_{\widehat{H}}$ and from $C_{_{F^n_{jj'}}}(u,u)=\es(C_{_{F_{jj'}}}(u^{1/n},u^{1/n})\di)^n$. 
We will then have
\begin{eqnarray}\nn
\dst \lim_{u\uparrow 1^+}\lim_{n\to\infty}\frac{\es(C_{_{F_{jj'}}}(u^{1/n},u^{1/n})\di)^n}{C_{_{F_{jj'}}}(u,u)}
=\lim_{n\to\infty}\lim_{u\uparrow 1^+}\frac{\es(C_{_{F_{jj'}}}(u^{1/n},u^{1/n})\di)^n}{C_{_{F_{jj'}}}(u,u)}=1,
\end{eqnarray}
which guarantees the constancy of $\chi^{F^n}_{jj'}$ and $\overline{\chi}^{F^n}_{jj'}$, as $n\to\infty$.\\


The presence of dependence among the variables of $\{\mb{X}_n\}$  expressed by a function $\theta(\pmb{\tau})$ with values less than one, may affect the limiting behavior of  $\chi^{F_n}_{jj'}$ but not the limiting behavior of $\overline{\chi}^{F_n}_{jj'}$, where $F_{n}$ denotes de df of $\mb{M}_n$.

\begin{pro}\label{p5}
For stationary sequences $\{\mb{X}_n\}$, with multivariate extremal index $\theta(\pmb{\tau})$, $\pmb{\tau}\in \er^d_+$, for any choice $j<j'$ in $D$, we have, $\overline{\chi}^{H}_{jj'}=\overline{\chi}^{\widehat{H}}_{jj'}$.
\end{pro}
\begin{proof}
Based on the spectral representation of MEV copulas and relation
\begin{eqnarray}\label{CHCHind}
\dst C_{_{H_{jj'}}}(u_j,u_{j'})=\es(C_{_{\widehat{H}_{jj'}}}\es(u_j^{1/\theta_j},u_{j'}^{1/\theta_{j'}}\di)\di)
^{\theta\es(-\frac{\log u_j}{\theta_j},-\frac{\log u_{j'}}{\theta_{j'}}\di)},
\end{eqnarray}
we have
\begin{eqnarray}\nn
\begin{array}{rl}
\overline{\chi}^{\widehat{H}}_{jj'}=&\dst \lim_{u\uparrow 1^+}\frac{2\log (1-u)}{\log\es(1-2u-C_{_{\widehat{H}_{jj'}}}(u,u)\di)}-1\\\\
=&\dst \lim_{u\uparrow 1^+}\frac{2\log (1-u)}{\log\es(1-2u-\exp\es(-\int_{0}^{1}\es(w(-\log u) \vee (1-w)(-\log u)\di) d\widehat{W}(w)\di)\di)}-1\\\\
=&\dst \lim_{u\uparrow 1^+}\frac{2\log (1-u)}{\log\es(1-2u-u^{-\log C_{_{\widehat{H}_{jj'}}}\es(e^{-1},e^{-1}\di)}\di)}-1
\end{array}
\end{eqnarray}
where $\widehat{W}$ is the spectral measure of $\widehat{H}$. On the other hand
\begin{eqnarray}\nn
\begin{array}{rl}
\overline{\chi}^{H}_{jj'}=&\dst \lim_{u\uparrow 1^+}\frac{2\log (1-u)}{\log\es(1-2u-u^{\theta_{jj'}\es(\frac{1}{\theta_j},\frac{1}{\theta_{j'}}\di)\es(-\log C_{_{\widehat{H}_{jj'}}}\es(\exp\es(-\theta_{j}^{-1}\di),\exp\es(-\theta_{j'}^{-1}\di)\di)\di)}\di)}-1
\end{array}
\end{eqnarray}
Therefore,
\begin{eqnarray}\nn
(1-\overline{\chi}^{H}_{jj'})=(1-\overline{\chi}^{\widehat{H}}_{jj'})A
\end{eqnarray}
with
\begin{eqnarray}\nn
\begin{array}{rl}
A=&\dst \lim_{u\uparrow 1^+}\frac{\log\es(1-2u-u^{\Gamma(1,1)}\di)}{\log\es(1-2u-
u^{\theta_{jj'}\es(\frac{1}{\theta_j},\frac{1}{\theta_{j'}}\di)
\Gamma\es(\frac{1}{\theta_j},\frac{1}{\theta_{j'}}\di)}\di)}\\\\
=&\dst \lim_{u\uparrow 1^+}\frac{\log\es(1-2u-u^{a}\di)}{\log\es(1-2u-
u^{b}\di)}=\dst \lim_{u\uparrow 1^+}\frac{-2+au^{a-1}}{-2+bu^{b-1}}\dst \lim_{u\uparrow 1^+}\frac{1-2u+u^{b}}{1-2u+u^{a}}=1,
\end{array}
\end{eqnarray}
with $a=\Gamma(1,1)$ and $b=\theta_{jj'}\es(\frac{1}{\theta_j},\frac{1}{\theta_{j'}}\di)
\Gamma\es(\frac{1}{\theta_j},\frac{1}{\theta_{j'}}\di)$.
\end{proof}

\begin{pro}\label{p4}
For stationary sequences $\{\mb{X}_n\}$, with multivariate extremal index $\theta(\pmb{\tau})$, $\pmb{\tau}\in \er^d_+$, we have, for any choice $j<j'$ in $D$,
\begin{enumerate}
\item[(a)] $\chi^{H}_{jj'}=2-\theta_{jj'}\es(\frac{1}{\theta_j},\frac{1}{\theta_{j'}}\di)\Gamma_{jj'}\es(\frac{1}{\theta_j},\frac{1}{\theta_{j'}}\di)$;
\item[(b)] $\chi^{H}_{jj'}-\chi^{\widehat{H}}_{jj'}=\Gamma_{jj'}\es(1,1\di)-\theta_{jj'}\es(\frac{1}{\theta_j},\frac{1}{\theta_{j'}}\di)\Gamma_{jj'}\es(\frac{1}{\theta_j},\frac{1}{\theta_{j'}}\di)$.
\end{enumerate}
\end{pro}
\begin{proof}
Using the spectral representation of MEV copulas and relation (\ref{CHCHind}),
we have
\begin{eqnarray}\nn
\begin{array}{rl}
\chi^{H}_{jj'}=&\dst 2-\theta_{jj'}\es(\frac{1}{\theta_j},\frac{1}{\theta_{j'}}\di)
\lim_{u\uparrow 1^+}\frac{\int_{0}^{1}\es(-\frac{\log u w}{\theta_j}\vee-\frac{\log u (1-w)}{\theta_{j'}}\di)d\widehat{W}(w)}{-\log u}\\\\
=&\dst 2-\theta_{jj'}\es(\frac{1}{\theta_j},\frac{1}{\theta_{j'}}\di)
\int_{0}^{1}\es(\frac{w}{\theta_j}\vee\frac{1-w}{\theta_{j'}}\di)d\widehat{W}(w)\\\\
=&\dst 2-\es(-\theta_{jj'}\es(\frac{1}{\theta_j},\frac{1}{\theta_{j'}}\di)
\log C_{_{\widehat{H}_{jj'}}}\es(\exp(-1/\theta_{j}),\exp(-1/\theta_{j'}) \di)\di)\\\\
=&\dst 2-\theta_{jj'}\es(\frac{1}{\theta_j},\frac{1}{\theta_{j'}}\di)
\Gamma_{jj'}\es(\frac{1}{\theta_j},\frac{1}{\theta_{j'}}\di),
\end{array}
\end{eqnarray}
where $\widehat{W}$ is the spectral measure of $\widehat{H}$.
\end{proof}

The previous proposition can be rewritten in terms of the extremal coefficients $\varepsilon^{H}_{jj'}$ and $\varepsilon^{\widehat{H}}_{jj'}$, such that, $C_{_{\widehat{H}_{jj'}}}(u,u)=u^{\varepsilon^{\widehat{H}}_{jj'}}$  and $C_{_{H_{jj'}}}(u,u)=u^{\varepsilon^{H}_{jj'}}$, since these satisfy the relations $\chi^{H}_{jj'}=2-\varepsilon^{H}_{jj'}$ and $\chi^{\widehat{H}}_{jj'}=2-\varepsilon^{\widehat{H}}_{jj'}$. From (a) we conclude that $\varepsilon^{H}_{jj'}=\theta_{jj'}\es(\frac{1}{\theta_j},\frac{1}{\theta_{j'}}\di)\Gamma_{jj'}\es(\frac{1}{\theta_j},\frac{1}{\theta_{j'}}\di)$. Consequently, for the measure of asymptotic independence called madogram (Naveau \ea~\cite{nav+09} 2009), defined by
\begin{eqnarray}\nn
\nu_{jj'}^{F}=\frac{1}{2}E\es|F_j(X_{nj})-F_{j'}(X_{nj'})\di|
\end{eqnarray}
and satisfying
\begin{eqnarray}\nn
\nu_{jj'}^{F}=\frac{1}{2}\frac{\varepsilon^{F}_{jj'}-1}{\varepsilon^{F}_{jj'}+1},
\end{eqnarray}
we have
\begin{enumerate}
\item[(a)] $\nu_{jj'}^{F}=\nu_{jj'}^{\widehat{H}}=\frac{1}{2}\frac{\Gamma_{jj'}(1,1)-1}{\Gamma_{jj'}(1,1)+1}$;
\item[(b)] $\nu_{jj'}^{H}=\frac{1}{2}
    \frac{\theta_{jj'}\es(\frac{1}{\theta_j},\frac{1}{\theta_{j'}}\di)\Gamma_{jj'}\es(\frac{1}{\theta_j},\frac{1}{\theta_{j'}}\di)-1}
    {\theta_{jj'}\es(\frac{1}{\theta_j},\frac{1}{\theta_{j'}}\di)\Gamma_{jj'}\es(\frac{1}{\theta_j},\frac{1}{\theta_{j'}}\di)+1}.$\\
\end{enumerate}

Therefore, for large $n$, the madogram of $(M_{nj},M_{nj'})$ can not be taken by the madogram of $(\widehat{M}_{nj},\widehat{M}_{nj'})$.\\

From relation (b) in Proposition \ref{p3}, we conclude that
\begin{eqnarray}\label{chimin}
\chi_{jj'}^{H}\geq \theta_{jj'}^{**}\tau_{jj'}^{**}\es(\frac{1}{\theta_j\vee \theta_{j'}}\di)
\end{eqnarray}
and we can establish the following consequence about the value of the difference between $\chi_{jj'}^{H}$ and $\chi_{jj'}^{\widehat{H}}$.

\begin{cor}
For stationary sequences $\{\mb{X}_n\}$ satisfying the strong-mixing condition, with multivariate extremal index $\theta(\pmb{\tau})$, $\pmb{\tau}\in \er^d_+$, we have, for any choice $j<j'$ in $D$,
\begin{enumerate}
\item[(a)] $\theta(\pmb{\tau})=\theta$,  $\forall \pmb{\tau}\in \er^d_+$ implies $\chi_{jj'}^{H}=\chi_{jj'}^{\widehat{H}}$;
\item[(b)] $\es|\chi^{H}_{jj'}-\chi^{\widehat{H}}_{jj'}\di|
    \geq\max\es\{\theta^{**}_{jj'}\tau^{**}_{jj'}\es(\frac{1}{\theta_j\vee\theta_{j'}}\di)-
    2+\Gamma_{jj'}\es(1,1\di),1-\Gamma_{jj'}\es(1,1\di)\di\}$.
\end{enumerate}
\end{cor}
\begin{proof}
\begin{enumerate}
\item[(a)] If $\theta(\pmb{\tau})$ is constant equal to $\theta$, then $\theta_j=\theta_{j'}=\theta$ and, since $\Gamma$ is homogeneous of order 1,  from (b) of Proposition \ref{p4}, we have $\chi_{jj'}^{H}-\chi_{jj'}^{\widehat{H}}=\Gamma_{jj'}\es(1,1\di)
    -\Gamma_{jj'}\es(\frac{\theta}{\theta},\frac{\theta}{\theta}\di)=0$;
\item[(b)] The inequality follows from (b) of Proposition \ref{p4} and from (\ref{chimin}).
\end{enumerate}
\end{proof}

We emphasize that the quantity $ \theta_{jj'}^{**}\tau_{jj'}^{**}\es(\frac{1}{\theta_j\vee \theta_{j'}}\di)$ that we find in (\ref{chimin}) and in (b)  of the previous proposition reflects a clustering propensity of $X_{nj}\wedge X_{nj'}$ through the extremal index $\theta^{**}_{jj'}$ and
\begin{eqnarray}\nn
\tau^{**}_{jj'}\es(\frac{1}{\theta_j\vee \theta_{j'}}\di)=\lim_{n\to\infty}nP\es(X_{nj}>n(\theta_j\vee \theta_{j'}),X_{nj'}>n(\theta_j\vee \theta_{j'})\di).
\end{eqnarray}

From this discussion we conclude that:
\begin{enumerate}
\item[(i)] The tail dependencies of  $\es(\widehat{M}_{n1},\widehat{M}_{n2}\di)$ and of $\es(M_{n1},M_{n2}\di)$, for large $n$, evaluated through coefficient $\chi$, can be considered equal when the multivariate extremal index is constant, otherwise they differ in at least
$\max\es\{\theta^{**}_{jj'}\tau^{**}_{jj'}\es(\frac{1}{\theta_j\vee\theta_{j'}}\di)-
    2+\Gamma_{jj'}\es(1,1\di),1-\Gamma_{jj'}\es(1,1\di)\di\}$, where the previous quantities can be estimated from the existing methods in literature.
\item[(ii)] If we estimate the dependence $\chi_{jj'}^{F}$ on the tail of $\es(X_{nj}, X_{nj'}\di)$, we do not obtain the dependence on the tail of $\es(M_{n1},M_{n2}\di)$, unless we correct the result with an estimate of $\Gamma_{jj'}\es(1,1\di)-
    \theta_{jj'}\es(\frac{1}{\theta_j},\frac{1}{\theta_{j'}}\di)\Gamma_{jj'}\es(\frac{1}{\theta_j},\frac{1}{\theta_{j'}}\di)$.\\
\end{enumerate}

In cases where $\widehat{H}$ has totally dependent marginals ($\chi_{jj'}^{\widehat{H}}=1$) or has independent marginals ($\chi_{jj'}^{\widehat{H}}=0$), the previous lower limit loses interest by triviality. We underline the expression of $\chi_{jj'}^{H}$ in these two cases in the next result, which is derived from (a) of Proposition \ref{p4}.

\begin{cor}
For stationary sequences $\{\mb{X}_n\}$, with multivariate extremal index $\theta(\pmb{\tau})$, $\pmb{\tau}\in \er^d_+$, we have, for any choice $j<j'$ in $D$,
\begin{enumerate}
\item[(a)] If $H$ has independent marginals, then $\chi^{H}_{jj'}=2-\es(\frac{1}{\theta_j}+\frac{1}{\theta_{j'}}\di)
    \theta_{jj'}\es(\frac{1}{\theta_j},\frac{1}{\theta_{j'}}\di)$;
\item[(b)] If $H$ has totally dependent marginals, then $\chi^{H}_{jj'}=2-\es(\frac{1}{\theta_j}\vee\frac{1}{\theta_{j'}}\di)
    \theta_{jj'}\es(\frac{1}{\theta_j},\frac{1}{\theta_{j'}}\di)$.\\
\end{enumerate}
\end{cor}

Now we construct some examples that illustrate the cases $\chi_{jj'}^{H}>\chi_{jj'}^{\widehat{H}}$ and $\chi_{jj'}^{H}<\chi_{jj'}^{\widehat{H}}$.\\

\begin{ex}
We first consider the following bivariate M4 process with one moving pattern,
$$
\es\{\begin{array}{c}
X_{n1}=\frac{1}{8}Z_{n-1}\vee \frac{1}{8}Z_{n}\vee\frac{6}{8}Z_{n+1}\\
X_{n2}=\frac{2}{8}Z_{n-1}\vee \frac{1}{8}Z_{n}\vee\frac{5}{8}Z_{n+1},
\end{array}\di.
$$
where $Z_{n}\equiv Z_{1,n}$, $\forall \,n\geq 1$. We have in this case
$$
C_F(u_1,u_2)=\es(u_1^{1/8}\wedge u_2^{2/8}\di)\es(u_1^{1/8}\wedge u_2^{1/8}\di)\es(u_1^{6/8}\wedge u_2^{5/8}\di)
$$
and
$$
\chi^F=\chi^{\widehat{H}}=2-\es(\frac{2}{8}+\frac{1}{8}+\frac{6}{8}\di)=\frac{7}{8}.
$$
Otherwise
$$
H(x_1,x_2)=\exp\es(-\es(\frac{6x_1^{-1}}{8}\vee \frac{5x_2^{-1}}{8}\di)\di).
$$
Therefore, $C_H(u_1,u_2)=u_1\wedge u_2$ and $\chi^H=1>\chi^{\widehat{H}}$.
\end{ex}

\begin{ex}
Now consider a modification in the above example through the introduction of one more pattern,
$$
\es\{\begin{array}{c}
X_{n1}=\frac{1}{8}Z_{1,n}\vee \frac{6}{8}Z_{1,n+1}\vee\frac{1}{8}Z_{2,n}\\
X_{n2}=\frac{1}{8}Z_{1,n}\vee \frac{5}{8}Z_{1,n+1}\vee\frac{2}{8}Z_{2,n}
\end{array}\di.,
$$
We have the same $C_F$ and $\chi^F=\frac{7}{8}$ as in the previous example, but here
$$
H(x_1,x_2)=\exp\es(-\es(\frac{6x_1^{-1}}{8}\vee \frac{5x_2^{-1}}{8}\di)\di)\exp\es(-\es(\frac{x_1^{-1}}{8}\vee \frac{2x_2^{-1}}{8}\di)\di).
$$
and therefore,
$$C_H(u_1,u_2)=\es(u_1^{6/7}\wedge u_2^{5/7}\di)\es(u_1^{1/7}\wedge u_2^{2/7}\di).$$
Then $\chi^H=2-\es(\frac{6}{7}+\frac{2}{7}\di)=\frac{6}{7}<\chi^{\widehat{H}}$.
\end{ex}

%

\end{document}